\documentclass{amsart}
\usepackage{graphicx,amsfonts,amssymb,amsmath,amsthm,url,
  verbatim}
\usepackage[normalem]{ulem}

\newcommand{\Tr}{\mathrm{Tr}}

\newcommand{\Sp}{\mathrm{Sp}}

\newcommand{\T}[1]{{}^t{{#1}}}

\newcommand{\A}{{\mathbb A}}
\renewcommand{\a}{{\mathfrak a}}
\newcommand{\Q}{{\mathbb Q}}
\newcommand{\Z}{{\mathbb Z}}

\newcommand{\GL}{{\rm GL}}

\newcommand{\SL}{{\rm SL}}

\newcommand{\GSp}{{\rm GSp}}
\newcommand{\PGSp}{{\rm PGSp}}

\newcommand{\disc}{{\rm disc}}

\DeclareMathOperator{\Cl}{Cl}
\newcommand{\mat}[4]
{{\setlength{\arraycolsep}{0.5mm}\left
(\begin{array}{cc}#1&#2\\#3&#4\end{array}\right)}}

\newcommand{\forget}[1]{}

\newtheorem{lemma}{Lemma}[section]
\newtheorem{theorem}{Theorem}

\newtheorem{conjecture}[lemma]{Conjecture}
\theoremstyle{remark}
\newtheorem{remark}[lemma]{Remark}

\begin{document}

\bibliographystyle{plain}

\author{Abhishek Saha}
\address{ Departments of Mathematics \\
  University of Bristol \\
  Bristol, United Kingdom} \email{abhishek.saha@bristol.ac.uk}

\title{A relation between multiplicity one and B\"ocherer's conjecture}

\begin{abstract} We show that a weak form of the generalized B\"ocherer's conjecture implies multiplicity one for Siegel cusp forms of degree 2.

\end{abstract}

 \maketitle

\section{Multiplicity one}

 Let $S_k(\Gamma)$   denote the space of classical holomorphic cusp forms of weight $k$ for $\Gamma = \SL_2(\Z)$. Suppose that $f_1$ and $f_2$ are two Hecke eigenforms in $S_k(\Gamma)$ such that the  Hecke eigenvalues $\lambda_p(f_1)$ and $\lambda_p(f_2)$ are equal for all primes $p$. Then, it is well known that there exists a constant $c$ such that $f_1 = c f_2$. Indeed, the various Hecke eigenvalues $\lambda_p(f_i)$ determine all the Fourier coefficients of $f_i$, hence the form $f_i$ itself, up to a common multiple.

 The above fact is a special case of the ``multiplicity one theorem" for $\GL_2$ due to Jacquet--Langlands. Multiplicity one was extended to the case of cuspidal representations on the group $\GL_n$ independently by Shalika~\cite{shalmultone} and Piatetski-Shapiro~\cite{psmultone}.  However, much less is known for cuspidal representations on other reductive groups.

 In this note, we consider the space\footnote{For definitions and background on Siegel cusp forms, see~\cite{klingen}.} $S_k(\Gamma_2)$ of holomorphic Siegel cusp forms of weight $k$ for $\Gamma_2 = \Sp_4(\Z)$. These forms give rise to cuspidal representations on the group $\GSp_4$, for which multiplicity one remains unknown. In particular, one has the following conjecture.

 \begin{conjecture}[Multiplicity one for Siegel cusp forms of degree 2 and full level]\label{con:mult} Let $f_1$ and $f_2$ be two Hecke eigenforms in $S_k(\Gamma_2)$ such that for all primes $p$, we have an equality of Hecke eigenvalues $\lambda_p(f_1) = \lambda_p(f_2)$ and $\lambda_{p^2}(f_1) = \lambda_{p^2}(f_2)$. Then, there exists a constant $c$ such that $f_1 = c f_2$.

 \end{conjecture}

  The above conjecture is deep. For instance, it does not appear to follow from the transfer to $\GL_4$ that was proved in~\cite{transfer} even though multiplicity one is known for $\GL_4$. We can get some idea about the difficulties involved by looking at the proof of multiplicity one for cusp forms on $\GL_n$. The proof involves combining the Whittaker expansion of cusp forms on $\GL_n$ with the uniqueness of Whittaker models. This approach does not work for holomorphic Siegel cusp forms because the corresponding Whittaker expansion does not exist (Hecke eigenforms in $S_k(\Gamma_2)$ are non-generic).

  It seems appropriate to point out here that in Conjecture~\ref{con:mult}, one may replace ``for all primes" by ``for almost all primes" without increasing the difficulty of the problem. Let us briefly explain this point. Hecke eigenforms in $S_k(\Gamma_2)$ that are not Saito-Kurokawa lifts (the Saito-Kurokawa lifts can be dealt with separately) lead to cuspidal representations of $\GSp_4$ whose local components at all finite places are tempered unramified principal series. If two eigenforms have the same eigenvalues at almost all primes, then they lead to representations that are nearly equivalent (i.e., their local components are equivalent at almost all primes). A fairly simple argument using the the global functional equation and the temperedness now shows that the two representations must have the same local $L$-function at all the remaining primes. Since unramified principal series representations are determined by their $L$-functions, it follows that the two representations are in fact equivalent. So the eigenforms we started off with must have the same eigenvalues at \emph{all} primes. This argument does not work if $\Gamma_2$ is replaced with a congruence  subgroup. Indeed, there are many examples of cuspidal representations on $\GSp_4$ that are nearly equivalent but not equivalent, e.g. those provided by the various Saito-Kurokawa lifts~\cite{ps-sk}, the various Yoshida lifts~\cite{sahaschmidt}, and the CAP representations of Borel type~\cite{howe-ps, filebighash}.

  \section{B\"ocherer's conjecture}
 Another deep and famous conjecture for Siegel cusp forms of degree 2 deals with the relation between the central $L$-values and the Fourier coefficients. Before stating this conjecture, we briefly recall the Fourier expansion for Siegel cusp forms. For any $f \in S_k(\Gamma_2)$, we can write

 $$f(Z)
=\sum_{S } a(f, S) e^{2 \pi i \Tr(SZ)},$$ where the Fourier coefficients $a(f,S)$ are indexed by matrices $S$ of the form
\begin{equation}\label{e:matrixform}
 S=\mat{a}{b/2}{b/2}{c},\qquad a,b,c\in\Z, \qquad a>0, \qquad \disc(S) := b^2 - 4ac < 0.
 \end{equation}
 Equivalently, the Fourier coefficients $a(f,S)$ are indexed by all positive, integral, binary quadratic forms. In fact, using the defining relation for Siegel cusp forms, we can see  that \begin{equation}\label{siegelinv}
a(f, \T{A}SA) = a(f,S)
\end{equation}
for all $A \in \SL_2(\Z)$, thus showing that $a(f, S)$ only depends on the $\SL_2(\Z)$-equivalence class of the matrix $S$, or equivalently, only on the proper equivalence class of the associated binary quadratic form.

Let $d > 0$ be an integer such that $-d$ is a fundamental discriminant.\footnote{Recall that an integer $n$ is a fundamental discriminant if \emph{either} $n$ is a squarefree integer congruent to 1 modulo 4 \emph{or} $n = 4m$ where $m$ is a squarefree integer congruent to 2 or 3 modulo 4.} Put $K = \Q(\sqrt{-d})$ and let $\Cl_K$ denote the ideal class group of $K$. It is a fact going back to Gauss that the $\SL_2(\Z)-$equivalence classes of binary quadratic forms of discriminant $-d$ are in natural bijective correspondence with the elements of $\Cl_K$. In view of the comments above, it follows that for any $f \in S_k(\Gamma_2)$ and any $c \in \Cl_K$ the notation $a(f, c)$ makes sense. We define

\begin{equation}\label{rfddef}
R(f, K) = \sum_{c \in \Cl_K}a(f, c).
\end{equation}

Now, suppose that $f$
 is an eigenform for the local Hecke algebras at all places. Then $f$ gives rise to an irreducible cuspidal automorphic representation $\pi_f$ of $\GSp_4(\A)$; see \cite{NPS}. The remarkable conjecture below was first made by B\"ocherer~\cite{boch-conj}.

 \begin{conjecture}[B\"ocherer's conjecture]\label{bochconj} Let $f \in S_k(\Gamma_2)$ be a non-zero Hecke eigenform and $\pi_f$ the associated automorphic representation. Then there exists a constant $c_f$ depending only on $f$ such that for any imaginary quadratic field $K=\Q(\sqrt{-d})$ with $-d$ a fundamental discriminant, we have

 $$ |R(f, K)|^2 = c_f \cdot d^{k-1} \cdot w(K)^{2} \cdot L(1/2, \pi_f \times \chi_{-d}) .$$

 Above, $\chi_{-d}$ is the quadratic Hecke character associated via class field theory to the field $\Q(\sqrt{- d})$, $w(K)$ denotes the number of distinct roots of unity inside $K$, and $L(s, \pi_f \times \chi_{-d})$ denotes the associated Langlands $L$-function.

 \end{conjecture}

The above conjecture is deep and so far not proven for any Siegel cusp form that is not a lift of some sort. Theoretical evidence for the truth of a refined version of the above conjecture was provided in 
work of the author with Kowalski and Tsimerman
(see~\cite[(5.4.5)]{kst2}).

B\"ocherer's conjecture has been further generalized by various people. We note in particular the papers by Furusawa--Shalika~\cite{furusawa-shalika-fund}, Furusawa--Martin \cite{furusawa-martin} and Prasad--Takloo-Bighash~\cite{prasadbighash}, as well  as a conjecture of Dipendra Prasad adapting Ichino-Ikeda's conjecture to this setting~\cite{prasadbighashpreprint}. In these generalizations, one takes a linear combination of the Fourier coefficients $a(f, c)$ with the values taken by an ideal class character of $\Cl_K$ (more generally, a Hecke character on $\A_K^\times$).  Partial progress towards such a conjecture has been made in recent work of Furusawa and Martin. Such a formulation is also closely related to the global Gross-Prasad conjecture
  for $(SO(5), SO(2))$.

We now describe a ``weak version" of a specific refinement of  B\"ocherer conjecture. Let $d$, $K$, $f$, $\pi_f$ be as before. For any character $\Lambda$ of the finite group $\Cl_K$, we make the definition

\begin{equation}\label{rfdef}
R(f, K, \Lambda) = \sum_{c \in \Cl_K}a(f, c) \Lambda^{-1}(c).
\end{equation}

Also, define the theta-series $$\theta_\Lambda(z) = \sum_{0 \ne \a \subset O_K}\Lambda(\a) e^{2 \pi i N(\a)z}.$$
Thus, $\theta_\Lambda$ is a holomorphic modular form of weight $1$ and nebentypus $\left(\frac{-d}{*} \right)$ on $\Gamma_0(d)$; it is a cusp form if and only if $\Lambda^2 \ne 1$. The form $\theta_\Lambda$ generates a dihedral automorphic representation of $\GL_2(\A)$ which coincides with the automorphic induction of $\Lambda$. We let $L(s, \pi_f \times \theta_\Lambda)$ denote the Langlands $L$-function attached to the Rankin--Selberg convolution of this dihedral representation with $\pi_f$. We now state the

 \begin{conjecture}[Generalized B\"ocherer's conjecture, weak form]\label{weakgenboch} Let $f \in S_k(\Gamma_2)$ be a non-zero Hecke eigenform and $\pi_f$ the associated automorphic representation. Let $K=\Q(\sqrt{-d})$ with $-d$ a fundamental discriminant, and $\Lambda$ be a character of the ideal class group of $K$. Then $ R(f, K, \Lambda) \neq 0$ if and only if $L(1/2, \pi_f \times \theta_\Lambda) \neq 0.$
\end{conjecture}

\begin{remark} In the \emph{special case} $\Lambda = 1$, the quantity $R(F, K, \Lambda)$ is simply $R(F, K)$ in our earlier notation and the $L$-function $L(s, \pi_f \times \theta_\Lambda)$ factors as a product $L(s, \pi_f)L(s, \pi_f \times \chi_{-d})$. Thus, Conjecture~\ref{weakgenboch} is compatible with Conjecture~\ref{bochconj}.
\end{remark}

Conjecture~\ref{weakgenboch} is in the spirit of the global Gross-Prasad conjectures since it asserts that the non-vanishing of a (Bessel) period $R(F, K, \Lambda)$ is equivalent to the non-vanishing of the central value of a related $L$-function.  Indeed, Sugano's formula~\cite{sug} can be used to show that $R(F, K, \Lambda)$ is related to a global Bessel period, and as such in the case of trivial central character this is nothing but the global Gross-Prasad (note $\PGSp(4) \simeq SO(3, 2)$). For a somewhat stronger formulation, see~\cite[Conj. 1.11]{furusawa-shalika-fund}. In general, the new conjectures of Gan, Gross, and Prasad~\cite{ggp} are relevant.

\section{The main result}

\begin{theorem}\label{t:main1} Conjecture~\ref{weakgenboch} implies Conjecture~\ref{con:mult}.

\end{theorem}

\begin{proof} The proof relies on the following fact which was proved in \cite{sahafund}:

 \medskip

\emph{  Let $0 \neq f \in S_k(\Gamma_2)$. Then there exists a matrix $S$ of the form given by equation~\eqref{e:matrixform} such that $\disc(S)$ is a fundamental discriminant and $a(f,S) \neq 0$.}

\medskip

Now, assume Conjecture~\ref{weakgenboch}.  Suppose $f_1$ and $f_2$ are two Hecke eigenforms in $S_k(\Gamma_2)$ such that for all primes $p$, we have an equality of Hecke eigenvalues $\lambda_p(f_1) = \lambda_p(f_2)$ and $\lambda_{p^2}(f_1) = \lambda_{p^2}(f_2)$. Let $-d <0$ be a fundamental discriminant such that there exists a matrix $S$ with $\disc(S) = -d$ and $a(f_2, S) \neq 0$. Such a $d$ exists by the fact quoted above. Put $K = \Q(\sqrt{-d})$ and pick a character $\Lambda$ of $\Cl_K$ such that $R(f_2, K, \Lambda) \neq 0$. Now, put $g_1 = f_1 - \frac{R(f_1, K, \Lambda)}{R(f_2, K, \Lambda)} f_2.$

We claim that $g_1 = 0$. Suppose not. Then $g_1$ and $f_2$ are two non-zero Hecke eigenforms in $S_k(\Gamma_2)$ such that for all primes $p$, we have an equality of Hecke eigenvalues $\lambda_p(g_1) = \lambda_p(f_2)$ and $\lambda_{p^2}(g_1) = \lambda_{p^2}(f_2)$. Since the Hecke operators at $p$ and $p^2$ generate the full Hecke algebra, it follows that $L(s, \pi_{g_1} \times  \theta_\Lambda) = L(s, \pi_{f_2} \times  \theta_\Lambda)  $. However, by construction, we have $R(f_2, K, \Lambda) \neq 0$ and $R(g_1, K, \Lambda) = 0$. Since we are assuming Conjecture~\ref{weakgenboch} holds true, this means that $L(1/2, \pi_{g_1} \times  \theta_\Lambda) = 0$ and $L(1/2, \pi_{g_1} \times  \theta_\Lambda) \neq 0$. This is a contradiction. Thus $g_1 = 0$. Hence $f_1$ and $f_2$ are multiples of each other.

\end{proof}

\begin{remark}  The proof shows that Bessel coefficients can serve as a substitute for the missing Whittaker coefficients. Although our theorem was stated only for full level eigenforms on $\Sp_4(\Z)$, it should be possible to use the same idea and get similar results for other groups under certain assumptions. The conjectures of Gan, Gross and Prasad~\cite{ggp} should be relevant here.
\end{remark}
\bibliography{lfunction}

\end{document}